\documentclass[11pt]{amsart}
\pdfoutput=1
\usepackage{amsfonts, amstext, amsmath, amsthm, amscd, amssymb}
\usepackage{epsfig, graphicx, psfrag}
\usepackage{color}

\newtheorem{thm}{Theorem}[section]
\newtheorem{lem}[thm]{Lemma}

\newtheorem{cor}[thm]{Corollary}

\newtheorem{defn}[thm]{Definition}
\newtheorem{prob}[thm]{Problem}

\newtheorem*{mainthm}{Theorem \ref{thm:main}}
\newtheorem*{ack}{Acknowledgements}
\newtheorem*{corone}{Corollary \ref{cor:fibered}}
\newtheorem*{cormain}{Corollary \ref{cor:main}}

\def \Z{\mathbb{Z}}
\def \T{\mathcal{T}}

\def \L{\mathcal{L}}

\begin{document}

\title[Generalized crossing changes in satellite knots]{Generalized crossing changes in satellite knots}
\author[C. Balm]{Cheryl Balm}

\address[]{Department of Mathematics, Michigan State
University, East Lansing, MI, 48824}

\email[]{balmcher@math.msu.edu}
\thanks{ {Research supported by NSF grant DMS-1105843.
}}

\begin{abstract} We show that if $K$ is a satellite knot in the 3-sphere $S^3$ which admits a generalized cosmetic crossing change of order $q$ with $|q| \geq 6$, then $K$ admits a pattern knot with a generalized cosmetic crossing change of the same order.  As a consequence of this, we find that any prime satellite knot in $S^3$ which admits a torus knot as a pattern cannot admit a generalized cosmetic crossing change of order $q$ with $|q| \geq 6$.  We also show that if there is any knot in $S^3$ admitting a generalized cosmetic crossing change of order $q$ with $|q| \geq 6$, then there must be such a knot which is hyperbolic.

%
%
%
\end{abstract}

\maketitle

\bigskip


\section{Introduction}

One of the many easily stated yet still unanswered questions in knot theory is the following: when does a crossing change on a diagram for an oriented knot $K$ yield a knot which is isotopic to $K$ in $S^3$?  We wish to study this and similar questions without restricting ourselves to any particular diagram for a given  knot, so we will consider crossing changes in terms of crossing disks.  A \emph{crossing disk} for an oriented knot $K\subset S^3$ is an embedded disk $D\subset S^3$ such that $K$ intersects ${\rm int}(D)$ twice with zero algebraic intersection number (see Figure \ref{fig:crcircle}).  A crossing change on $K$ can be achieved by performing $(\pm 1)$-Dehn surgery of $S^3$ along the \emph{crossing circle} $L = \partial D$.  (See \cite{pros-soss} for details on crossing changes and Dehn surgery.)  More generally, if we perform $(-1/q)$-Dehn surgery along the crossing circle $L$ for some $q \in \Z - \{ 0 \}$, we twist $K$ $q$ times at the crossing circle in question.  We will call this an \emph{order-$q$ generalized crossing change}. Note that if $q$ is positive, then we give $K$ $q$ right-hand twists when we perform $(-1/q)$-surgery, and if $q$ is negative, we give $K$ $q$ left-hand twists.

A crossing of $K$ and its corresponding crossing circle $L$ are called \emph{nugatory} if $L$ bounds an embedded disk in $S^3 - \eta(K)$, where $\eta(K)$ denotes a regular neighborhood of $K$ in $S^3$.  Obviously, a generalized crossing change of any order at a nugatory crossing of $K$ yields a knot isotopic to $K$.  

\begin{defn}{\rm
A (generalized) crossing change on $K$ and its corresponding crossing circle are called \emph{cosmetic} if the crossing change yields a knot isotopic to $K$ and is performed at a crossing of $K$ which is \emph{not} nugatory.  }
\end{defn}

The following question, often referred to as the nugatory crossing conjecture, is Problem 1.58 on Kirby's list \cite{kirby}.

\begin{figure}
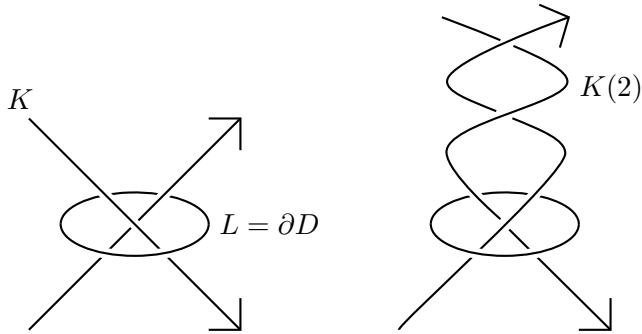
\caption{On the left is a crossing circle $L$ bounding a crossing disk $D$.  On the right is the knot resulting from an order-2 generalized crossing change at $L$.}\label{fig:crcircle}
\end{figure}

\begin{prob}\label{q:ncc}{\rm
Does there exist a knot $K$ which admits a cosmetic crossing change?  Conversely, if a crossing change on a knot $K$ yields a knot isotopic to $K$, must the crossing be nugatory?}
\end{prob}

Note that a crossing change (in the traditional sense) is the same as an order-$(\pm 1)$ generalized crossing change.  Hence, one can ask the following stronger question concerning cosmetic generalized crossing changes.

\begin{prob}\label{q:gncc}{\rm
Does there exist a knot $K$ which admits a cosmetic generalized crossing change of any order?}
\end{prob}

Problem \ref{q:ncc} was answered for the unknot by Scharlemann and Thompson when they showed that the unknot admits no cosmetic crossing changes in \cite{schar-thom} using work of Gabai \cite{gabai}.  Progress was made towards answering Problem \ref{q:ncc} for knots of braid index three by Wiley \cite{wiley}. Obstructions to cosmetic crossing changes in genus-one knots were found by the author with Friedl, Kalfagianni and Powell in \cite{balm} using Seifert forms and the first homology of the 2-fold cyclic cover of $S^3$ over the knot.  In particular, it was shown there that twisted Whitehead doubles of non-cable knots do not admit cosmetic crossing changes.

With regard to the stronger question posed in Problem \ref{q:gncc}, it has been shown by Kalfagianni that the answer to this question is no for fibered knots \cite{kalf} and by Torisu that the answer is no for 2-bridge knots \cite{torisu}.  The answer is also shown to be no for closed 3-braids and certain Whitehead doubles by the author and Kalfagianni in \cite{balm3}.  Torisu reduces Problem \ref{q:gncc} to the case where $K$ is a prime knot in \cite{torisu}.  

In this paper we will primarily be concerned with satellite knots, which are defined in the following way.

\begin{defn}{\rm
A knot $K$ is a \emph{satellite knot} if $M_{K} = \overline{S^3 - \eta(K)}$ contains a torus $T$ which is incompressible and not boundary-parallel in $M_{K}$ and such that $K$ is contained in the solid torus $V$ bounded by $T$ in $S^3$.  Such a torus $T$ is called a \emph{companion torus} for $K$.  Further, there exists a homeomorphism $f:(V', K') \to (V,K)$ where $V'$ is an unknotted solid torus in $S^3$ and $K'$ is contained in $\rm{int}(V')$.  The knot $K'$ is called a \emph{pattern knot} for the satellite knot $K$.  Finally, we require that $K$ is neither the core of $V$ nor contained in a 3-ball $B^3 \subset V$, and likewise for $K'$ in $V'$.}
\end{defn}

We may similarly define a \emph{satellite link}.  Note that a patten knot for a given satellite knot may not be unique, but it is unique once we have specified the companion torus $T$ and the map $f$.  In general, a torus $T$ in any orientable 3-manifold $N$ is called \emph{essential} if $T$ is incompressible and not boundary-parallel in $N$.  

The purpose of this paper is to prove the following.

\begin{mainthm}
Suppose $K$ is a satellite knot which admits a cosmetic generalized crossing change of order $q$ with $|q| \geq 6$.  Then $K$ admits a pattern knot $K'$ which also has an order-$q$ cosmetic generalized crossing change.
\end{mainthm}

This leads to the following two corollaries with regards to cosmetic generalized crossing changes.

\begin{corone} 
Suppose $K'$ is a torus knot.  Then no prime satellite knot with pattern $K'$ admits an order-$q$ cosmetic generalized crossing change with $|q| \geq 6$.
\end{corone}

\begin{cormain}
If there exists a knot admitting a cosmetic generalized crossing change of order $q$ with $|q|\geq 6$, then there must be such a knot which is hyperbolic.
\end{cormain}

Thus we have reduced Problem \ref{q:gncc} to the cases where either the knot is hyperbolic or the crossing change has order $q$ with $|q| < 6$.


\begin{ack}{\rm
The author would like to thank her advisor, Effie Kalfagianni, for her help with this project.}
\end{ack}

\section{Preliminaries}
Given a 3-manifold $N$ and submanifold $F \subset N$ of co-dimension 1 or 2, $\eta (F)$ will denote a closed regular neighborhood of $F$ in $N$.  If $N$ is a 3-manifold containing a surface $\Sigma$, then by \emph{$N$ cut along $\Sigma$} we mean $\overline{N - \eta(\Sigma)}$.  For a knot or link $\L \subset S^3$, we define $M_{\L} = \overline{S^3 - \eta(\L)}$.  

Given a knot $K$ with a crossing circle $L$, let $K_L(q)$ denote the knot obtained via an order-$q$ generalized crossing change at $L$.  We may simply write $K(q)$ for $K_L (q)$ when there is no danger of confusion about the crossing circle in question.

\begin{lem}\label{lem:irred} Let $K$ be an oriented knot with a crossing circle $L$.  If $L$ is not nugatory, then $M_{K \cup L}$ is irreducible. 
\end{lem}

\begin{proof}  We will prove the contrapositive.  Suppose $M_{K \cup L}$ is reducible.  Then $M_{K \cup L}$ contains a separating 2-sphere $S$ which does not bound a 3-ball $B \subset M_{K \cup L}$.  So $S$ must separate $K$ and $L$ in $S^3$, and consequently, $L \subset S^3$ lies in a 3-ball disjoint from $K$.  Since $L$ is unknotted, $L$ bounds a disc in this 3-ball which is in the complement of $K$, and hence $L$ is nugatory.
\end{proof}

Recall that a  knot  $K$ is called \emph{algebraically slice} if it admits a Seifert surface $S$ such that the Seifert form $\theta: H_1(S) \times H_1(S) \to \Z$ vanishes on a half-dimensional summand of $H_1(S)$.  In particular, if $K$ is algebraically slice, then the Alexander polynomial $\Delta_K(t)$ is of the form $\Delta_K(t) \doteq f(t) f(t^{-1})$, where $f(t)\in \Z[t]$ is a linear polynomial and $\doteq$ denotes equality up to multiplication by a unit in ${\Z}[t, t^{-1}]$.  (See \cite{lick} for more details.)  With this in mind, we have the following lemma from \cite{balm}, which we will need in the proof of Corollary \ref{cor:sat}.

\begin{lem}\label{lem:g1}
Suppose $K$ is a genus-one knot which admits a cosmetic generalized crossing change of any order $q \in \Z - \{ 0 \}$. Then $K$ is algebraically slice.
\end{lem}

\begin{proof}
For $q = \pm 1$, this is Theorem 1.1(1) of \cite{balm}.  The proof given there is easily adapted to generalized crossing changes of any order.
\end{proof}

Fix a knot $K$ and let $L$ be a crossing circle for $K$.  Let $M(q)$ denote the 3-manifold obtained from $M_{K \cup L}$ via a Dehn filling of slope $(-1/q)$ along $\partial \eta (L)$.  So, for $q \in \Z - \{ 0 \},\ M(q) = M_{K(q)}$, and $M(0) = M_K$.  We will sometimes use $K(0)$ to denote $K \subset S^3$ when we want to be clear that we are considering $K \subset S^3$ rather than $K \subset M_L$.

Suppose there is some $q \in \Z$ for which $K_L (q)$ is a satellite knot.  Then there is a companion torus $T$ for $K(q)$ and, by definition, $T$ must be essential in $M(q)$.  This essential torus $T$ must occur in one of the following two ways.

\begin{defn}{\rm
Let $T \subset M(q)$ be an essential torus.  We say $T$ is \emph{Type 1} if $T$ can be isotoped into $M_{K \cup L} \subset M(q)$.  Otherwise, we say $T$ is \emph{Type 2}.  If $T$ is Type 2, then $T$ is the image of a punctured torus $(P, \partial P) \subset (M_{K \cup L}, \partial \eta (L))$ and each component of $\partial P$ has slope $(-1/q)$ on $\partial \eta (L)$.}
\end{defn} 

In general, let $\L$ be any knot or link in $S^3$ and let $\Sigma$ be a boundary component of $M_{\L}$.  If $(P, \partial P) \subset (M_{\L}, \Sigma)$ is a punctured torus, then every component of $\partial P$ has the same slope on $\Sigma$, which we call the \emph{boundary slope} of $P$.  

Suppose $C_1$ and $C_2$ are two non-separating simple closed curves (or boundary slopes) on a torus $\Sigma$.  Let $s_i$ be the slope of $C_i$ on $\Sigma$, and let $[C_i]$ denote the isotopy class of $C_i$ for $i=1,2$.  Then $\Delta (s_1, s_2)$ is the minimal geometric intersection number of $[C_1]$ and $[C_2]$. It is known that if $s_i$ is the rational slope $(1/q_i)$ for some $q_i \in \Z$ for $i=1,2$, then $\Delta (s_1, s_2) = |q_1 - q_2|$.   (See \cite{gordon} for more details.)  Note that we consider $\infty = (1/0)$ to be a rational slope.

Gordon \cite{gordon} proved the following theorem relating the boundary slopes of punctured tori in link complements.  In fact, Gordon proved a more general result, but we state the theorem here only for the case which we will need later in Section \ref{s:pfs}. 

\begin{thm}[Gordon, Theorem 1.1 of \cite{gordon}]\label{thm:gordon}
Let $\L$ be a knot or link in $S^3$ and let $\Sigma$ be a boundary component of $M_{\L}$.  Suppose $(P_1, \partial P_1)$ and $(P_2, \partial P_2)$ are punctured tori in $(M_{\L}, \Sigma)$ such that the boundary slope of $P_i$ on $\Sigma$ is $s_i$ for $i=1,2$.  Then $\Delta (s_1, s_2) \leq 5$.
\end{thm}

Now suppose $K$ is a knot contained in a solid torus $V \subset S^3$.  An embedded disk $D \subset V$ is called a \emph{meridian disk} of $V$ if $\partial D = D \cap \partial V$ is a meridian of $\partial V$.  We call $K$ \emph{geometrically essential} (or simply \emph{essential}) in $V$ if every meridian disk of $V$ meets $K$ at least once.   With this in mind, we have the following lemma of Kalfagianni and Lin \cite{kalf-lin}. 

\begin{lem}[Lemma 4.6 of \cite{kalf-lin}]\label{lem:LinV}
Let $V \subset S^3$ be a knotted solid torus such that $K \subset \rm{int}(V)$ is a knot which is essential in $V$ and $K$ has a crossing disk $D$ with $D \subset \rm{int}(V)$.  If $K$ is isotopic to $K(q)$ in $S^3$, then $K(q)$ is also essential in $V$.  Further, if $K$ is not the core of $V$, then $K(q)$ is also not the core of $V$.
\end{lem}

\begin{proof}
Suppose, by way of contradiction, that $K(q)$ is not essential in $V$.  Then there is a 3-ball $B \subset V$ such that $K(q) \subset B$.  This means that the winding number of $K(q)$ in $V$ is 0.

Let $S_1$ be a Seifert surface for $K$ which is  of minimal genus in $M_L$, where $L = \partial D$.  We may isotope $S_1$ so that $S_1 \cap D$ consists of a single curve $\alpha$ connecting the two points of $K \cap D$.   Then twisting $S_1$ at $L$ via a $(-1/q)$-Dehn filling on $\partial \eta(L)$ gives rise to a Seifert surface $S_2$ for $K(q)$.  Since $M_{K \cup L}$ is irreducible by Lemma \ref{lem:irred}, we may apply Gabai's Corollary 2.4 of \cite{gabai} to see that $S_1$ and $S_2$ are minimal-genus Seifert surfaces in $S^3$ for $K$ and $K(q)$, respectively.  

Since the winding number of $K(q)$ in $V$ is 0, $S_1 \cap \partial V = S_2 \cap \partial V$ is homologically trivial in $\partial V$.  For $i=1,2$, we can surger $S_i$ along disks and annuli in $\partial V$ which are bounded by curves in $S_i \cap \partial V$ to get new minimal genus Seifert surfaces $S_i' \subset \rm{int}(V)$.  Then $S_2'$ is incompressible and $V$ is irreducible, so we can isotope $S_2'$ into $\rm{int}(B)$.  Hence $\alpha$ and therefore $D$ can also be isotoped into $\rm{int}(B)$.  But then $K$ must not be essential in $V$, which is a contradiction.

Finally, if $K$ is not the core of $V$, then $\partial V$ is a companion torus for the satellite knot $K$.  Since a satellite knot cannot be isotopic to the core of its companion torus, $K(q)$ cannot be the core of $V$.
\end{proof}

Given a compact, irreducible, orientable 3-manifold $N$, let $\T$ be a collection of disjointly embedded, pairwise non-parallel, essential tori in $N$, which we will call an \emph{essential torus collection} for $N$.  By Haken's Finiteness Theorem (Lemma 13.2 of \cite{hempel}) the number $$\tau(N) = \max \{ |\T|\ |\ \T {\rm\ is\ an\ essential\ torus\ collection\ for\ }N \}$$ is well-defined and finite, where $|\T|$ denotes the number of tori in $\T$.  We will call such a collection $\T$ with $|\T| = \tau(N)$ a \emph{Haken system} for $N$.  Note that any essential torus $T \subset N$ is part of some Haken system $\T$.

Before moving on the the proofs of Theorem \ref{thm:main} and its corollaries, we state the following results of Motegi \cite{motegi} (see also \cite{knoten}) and McCullough \cite{mccullough} which we will need in the next section.

\begin{lem}[Motegi, Lemma 2.3 of \cite{motegi}]\label{thm:motegi}
Let $K$ be a knot embedded in $S^3$ and let $V_1$ and $V_2$ be knotted solid tori in $S^3$ such that the embedding of $K$ is essential in $V_i$ for $i=1,2$.  Then there is an ambient isotopy $\phi: S^3 \to S^3$ leaving $K$ fixed such that one of the following holds.
\begin{enumerate}
\item $\partial V_1 \cap \phi (\partial V_2) = \emptyset$.
\item There exist meridian disks $D$ and $D'$ for both $V_1$ and $V_2$ such that some component of $V_1$ cut along $(D \sqcup D')$ is a knotted 3-ball in some component of $V_2$ cut along $(D \sqcup D')$.
\end{enumerate}
\end{lem}

By a \emph{knotted 3-ball}, we mean a ball $B$ for which there is no isotopy which takes $B$ to the standardly embedded 3-ball while leaving $D$ and $D'$ fixed.  (See Figure \ref{fig:knotted3}.)

\begin{figure}
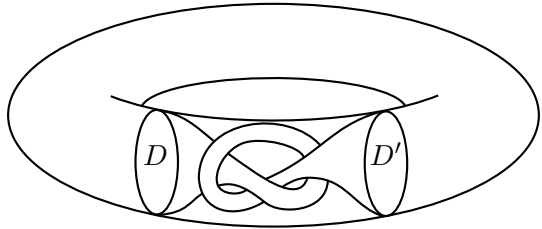
\caption{A knotted 3-ball $B$ inside of a solid torus with disks $D,D' \subset \partial B$.}\label{fig:knotted3}
\end{figure}

Before stating the result of McCullough, we recall the definition of a Dehn twist.  Let $C$ be a simple closed curve on a surface $\Sigma$ and let $A \subset \Sigma$ be an open annular neighborhood of $C$.  Then a \emph{Dehn twist} of $\Sigma$  at $C$ is a diffeomorphism $\phi: \Sigma \to \Sigma$ such that $\phi$ is the identity map on $\Sigma - A$ and $\phi|_{A}$ is given by a full twist around $C$.  More specifically, if $A = S^1 \times (0,1)$ is given the coordinates $(e^{i \theta},\, x)$, then $\phi|_A:(e^{i \theta},\, x) \to (e^{i (\theta + 2\pi x)},\, x)$.  The following result concerns diffeomorphisms of 3-manifolds that restrict to Dehn twists on the boundary of the manifold.

\begin{thm}[McCullough, Theorem 1 of \cite{mccullough}]\label{thm:dehntw}
Let $N$ be a compact, orientable 3-manifold that admits a homeomorphism which restricts to Dehn twists on the boundary of $N$ along a simple closed curve in $C \subset \partial N$ . Then $C$ bounds a disk in $N$.
\end{thm}

\section{Proofs of main results}\label{s:pfs}

The goal of this section is to prove Theorem \ref{thm:main} and its corollaries.  We begin with the following lemma.

\begin{lem}[Compare to Proposition 4.7 of \cite{kalf-lin}]\label{lem:type2orq<6}
Let $K$ be a prime satellite knot with a cosmetic crossing circle $L$ of order $q$.  Then at least one of the following must be true:
\begin{enumerate}
\item $M(q)$ contains no Type 2 tori
\item $|q| \leq 5$
\end{enumerate}
\end{lem}

\begin{proof}
Suppose $M(q)$ contains a Type 2 torus.  We claim that $M(0)$ must also contain a Type 2 torus.  Assuming this is true, $M(0)$ and $M(q)$ each contain a Type 2 torus and hence there are punctured tori $(P_0, \partial P_0)$ and $(P_q, \partial P_q)$ in $(M, \partial \eta (L))$ such that $P_0$ has boundary slope $\infty = (1/0)$ and $P_q$ has boundary slope $(-1/q)$ on $\partial \eta (L)$.  Then, by Theorem \ref{thm:gordon}, $\Delta( \infty, -1/q) = |q| \leq 5$, as desired.  Thus it remains to show that there is a Type 2 torus in $M(0)$.

Let $M = M_{K \cup L}$.  Since $L$ is not nugatory, Lemma \ref{lem:irred} implies that $M$ is irreducible and hence $\tau(M)$ is well-defined.  First assume that $\tau(M)=0$.  Since $K$ is a satellite knot, $M(0)$ must contain an essential torus, and it cannot be Type 1.  Hence $M(0)$ contains a Type 2 torus.


Now suppose that $\tau(M)>0$ and let $T$ be an essential torus in $M$.  Then $T$ bounds a solid torus $V \subset S^3$.  Let $\rm{ext}(V)$ denote $S^3 - V$.  If $K \subset \rm{ext}(V)$, then $L$ must be essential in $V$.  If $V$ were knotted, then either $L$ is the core of $V$ or $L$ is a sattelite knot with companion torus $V$.  This contradicts the fact that $L$ is unknotted.  Hence $T$ is an unknotted torus.  By definition, $L$ bounds a crossing disk $D$.  Since $D$ meets $K$ twice, $D \cap \textrm{ext}(V) \neq \emptyset$.  We may assume that $D$ has been isotoped (rel boundary) to minimize the number of components in $D \cap T$.  Since an innermost component of $D - (D \cap T)$ is a disk and $L$ is essential in the unknotted solid torus $V$, $D \cap T$ consists of standard longitudes on the unknotted torus $T$.  Hence $D \cap \textrm{ext}(V)$ consists of either one disk which meets $K$ twice, or two disks which each meet $K$ once.  In the first case, $L$ is isotopic to the core of $V$, which contradicts $T$ being essential in $M$.  In the latter case, the linking number $lk(K,V)= \pm 1$.  So $K$ can be considered as the trivial connect sum $K \# U$, where $U$ is the unknot and the crossing change at $L$ takes place in the unknotted  summand $U$.  (See Figure \ref{fig:KconnU}.)  The unknot does not admit cosmetic crossing changes of any order, so $K_L(q) = K \# K'$ where $K' \neq U$.  This contradicts the fact that $K_L(q) = K$.  Hence, we may assume that $T$ is knotted and $K$ is contained in the solid torus $V$ bounded by $T$.

\begin{figure}
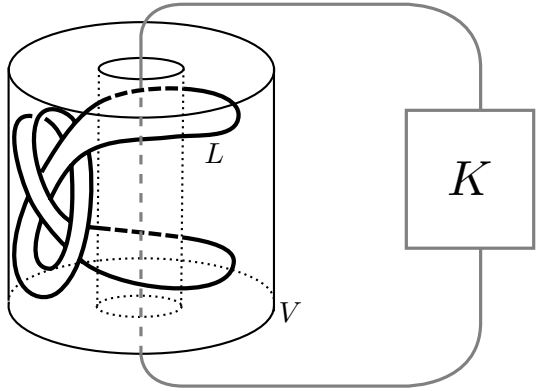
\caption{An example of an unknotted torus $V$ containing a crossing circle $L$ which bounds a crossing disk for the knot $K = K \# U$.}\label{fig:KconnU}
\end{figure}
 
If $L \subset \rm{ext}(V)$ and cannot be isotoped into $V$, then $D \cap T$ has a component $C$ that is both homotopically non-trivial and not boundary-parallel in $D - (K \cap D)$.  So $C$ must encircle exactly one of the two points of $K \cap D$.  This means that the winding number of $K$ in $V$ is $\pm{1}$.  Since $T$ cannot be boundary parallel in $M$, $K$ is not the core of $T$, and hence $T$ is a ``follow-swallow" torus for $K$ and $K$ is composite.  But this contradicts the assumption that $K$ is prime.  Hence we may assume that $L$ and $D$ are contained in $\rm{int}(V)$.  

Since $V$ is knotted and $D \subset \rm{int}(V)$, Lemma \ref{lem:LinV} implies that $T$ is also a companion torus for $K(q)$.  This means every Type 1 torus in $M(0)$ is also a Type 1 torus in $M(q)$.  Since $K(0)$ and $K(q)$ are isotopic, $\tau (M(0)) = \tau (M(q))$.  By assumption, $M(q)$ contains a Type 2 torus, which must give rise to a Type 2 torus in $M(0)$, as desired.

\end{proof}

The following is an immediate corollary of Lemma \ref{lem:type2orq<6}.

\begin{cor}\label{cor:taunot0}
Let $K$ be a prime satellite knot with a cosmetic crossing circle $L$ of order $q$ with $|q| \geq 6$.  Then $\tau (M_{K \cup L}) > 0$.
\end{cor}

We are now ready to prove our main theorem.

\begin{thm}\label{thm:main}
Suppose $K$ is a satellite knot which admits a cosmetic generalized crossing change of order $q$ with $|q| \geq 6$.  Then $K$ admits a pattern knot $K'$ which also has an order-$q$ cosmetic generalized crossing change.
\end{thm}

\begin{proof}
Let $K$ be such a satellite knot with a crossing circle $L$ bounding a crossing disk $D$ which corresponds to a cosmetic generalized crossing change of order $q$.  Let $M = M_{K \cup L}$.  

If $K$ is a composite knot, then Torisu \cite{torisu} showed that the crossing change in question must occur within one of the summands of $K = K_1 \# K_2$, say $K_1$.  We may assign to $K$ the ``follow-swallow" companion torus $T$, where the core of $T$ is isotopic to $K_2$.  Then the patten knot corresponding to $T$ is $K_1$ and the theorem holds.  

Now assume $K$ is prime.  By Corollary \ref{cor:taunot0}, $\tau (M) >0$.  Let $T$ be an essential torus in $M$ and let $V \subset S^3$ be the solid torus bounded by $T$ in $S^3$.  As shown in the proof of Lemma \ref{lem:type2orq<6}, $V$ is knotted in $M$ and $D$ can be isotoped to lie in $\rm{int}(V)$.  This means $T$ is a companion torus for the satellite link $K \cup L$.  Let $K' \cup L'$ be a pattern link for $K \cup L$ corresponding to $T$.  So there is an unknotted solid torus $V' \subset S^3$ such that $(K' \cup L') \subset V'$ and there is a homeomorphism $f: (V',K',L') \to (V,K,L)$.

Let $\T$ be a Haken system for $M$ such that $T \in \T$.  We will call a torus $J \in \T$ \emph{innermost with respect to $K$} if $M$ cut along $\T$ has a component $C$ such that $\partial C$ contains $\partial \eta (K)$ and a copy of $J$.  In other words, $J \in \T$ is innermost with respect to $K$ if there are no other tori in $\T$ separating $J$ from $\eta(K)$.  Choose $T$ to be innermost with respect to $K$.

Let $W= \overline{V - \eta(K \cup L)}$.  We first wish to show that $W$ is atoroidal.  By way of contradiction, suppose that there is an essential torus $F \subset W$.  Then $F$ bounds a solid torus in $V$ which we will denote by $\widehat{F}$.  Since $T$ is innermost with respect to $K$, either $F$ is parallel to $T$ in $M$, or $K \subset V - \widehat{F}$.  By assumption, $F$ is essential in $W$ and hence not parallel to $T \subset \partial W$.  So $K \subset V - \widehat{F}$ and, since $F$ is incompressible, $L \subset \widehat{F}$.  But then $F$ must be unknottted and parallel to $\partial \eta(L) \subset \partial W$, which is a contradiction.  Hence $W$ is indeed atoroidal, and $W' = \overline{V' - \eta(K' \cup L')}$ must be atoroidal as well.

To finish the proof, we must consider two cases, depending on whether $T$ is compressible in $V - \eta(K(q))$.

\smallskip
\smallskip

\noindent \emph{Case 1:} $K(q)$ is essential in $V$.

We wish to show that there is an isotopy $\Phi: S^3 \to S^3$ such that $\Phi(K(q)) = K(0)$ and $\Phi (V) = V$.  First, suppose $K(q)$ is the core of $T$.  By Lemma \ref{lem:LinV}, $K$ is also the core of $T$.  Since $L$ is cosmetic, there is an ambient isotopy $\psi: S^3 \to S^3$ taking $K(q)$ to $K(0)$.  Since $K(q)$ and $K(0)$ are both the core of $V$, we may choose $\psi$ so that $\psi(V) = V$ and let $\Phi = \psi$.

If $K(q)$ is not the core of $T$, then $T$ is a companion torus for $K(q)$.  Since $K(0) = (K(q))_L (-q)$, we may apply Lemma \ref{lem:LinV} to $K(q)$ to see that $T$ is also a companion torus for $K(0)$.  Again, there is an ambient isotopy $\psi: S^3 \to S^3$ taking $K(q)$ to $K(0)$ such that $V$ and $\psi(V)$ are both solid tori containing $K(0) = \psi(K(q)) \subset S^3$.  If $\psi(V) = V$, we once more let $\Phi = \psi$.  If $\psi(V) \neq V$, we may apply Lemma \ref{thm:motegi} to $V$ and $\psi(V)$.  If part (2) of Lemma \ref{thm:motegi} were satisified, then $\psi(V) \cap V$ would give rise to a knotted 3-ball contained in either $V$ or $\psi(V)$.  This contradicts the fact that $W$, and hence $\psi(W)$, are atoroidal.  Hence part (1) of Lemma \ref{thm:motegi} holds, and there is an isotopy $\phi: S^3 \to S^3$ fixing $K(0)$ such that $(\phi \circ \psi)(T) \cap T = \emptyset$.  Let $\Phi = (\phi \circ \psi):S^3 \to S^3$.  Recall that by Lemma \ref{lem:type2orq<6}, $M(q)$ contains no Type 2 tori.  Hence $T$ remains innermost with respect to $K(q)$ in $S^3$ and therefore $\Phi(T)$ is also innermost with respect to $K(0)$.   Either $T \subset S^3 - \Phi(V)$ or $\Phi(T) \subset S^3 - V$.  In either situation, the fact that  $T$ and $\Phi(T)$ are innermost implies that $T$ and $\Phi (T)$ are in fact parallel in $M_K$.  So, after an isotopy which fixes $K(0) \subset S^3$, we may assume that $\Phi(V) = V$.

Now let $h = (f^{-1} \circ \Phi \circ f):V' \to V'$.  Note that $h$ preserves the canonical longitude of $\partial V'$ (up to sign).  Since $h$ maps $K'(q)$ to $K'(0)$,  $K'(q)$ and $K'(0)$ are isotopic in $S^3$.  So either $L'$ gives an order-$q$ cosmetic generalized crossing change for the pattern knot $K'$, or $L'$ is a nugatory crossing circle for $K'$.

Suppose $L'$ is nugatory.  Then $L'$ bounds a crossing disk $D'$ and another disk $D'' \subset M_{K'}$.  We may assume $D' \cap D'' = L'$.  Let $A = D' \cup (D'' \cap V')$.  Since $\partial V'$ is incompressible in $V'- \eta(K')$, by surgering along components of of $D'' \cap \partial V'$ which bound disks or cobound annuli in $\partial V'$, we may assume $A$ is a properly embedded annulus in $V'$ and each component of  $A \cap \partial V'$ is a longitude of $V'$.  Since $L' \subset V'$, we can extend the homeomorphism $h$ on $V' \subset S^3$ to a homeomorphism $H$ on all of $S^3$.  Since $V'$ is unknotted, let $C$ be the core of the solid torus $S^3 - \rm{int} (V')$.  We may assume that $H$ fixes $C$.  Since $D' \cup D''$ gives the same (trivial) connect sum decomposition of $K'(0) = K'(q)$ and $H$ preserves canonical longitudes on $\partial V'$, we may assume $H(D')$ is isotopic to $D'$ and $H(D'')$ is isotopic to $D''$.  In fact, this isotopy may be chosen so that $H(C)$ and $H(V')$ remain disjoint throughout the isotopy and $H(V')=V'$ still holds after the isotopy.  Thus, we may assume $h(A)=A$ and $A$ cuts $V'$ into two solid tori $V'_1$ and $V'_2$, as shown in Figure \ref{fig:torus}.  

\begin{figure}
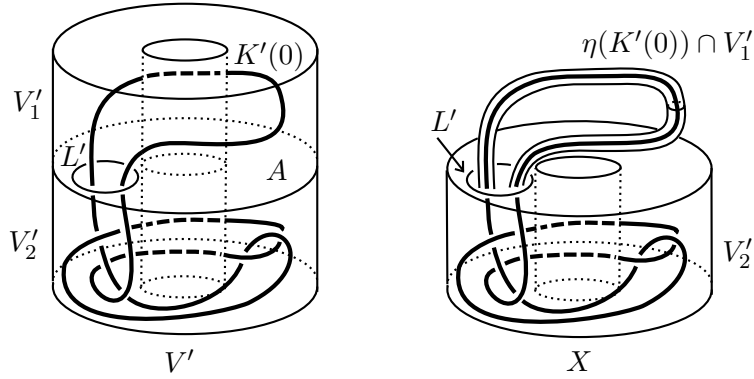
\caption{On the left is the solid torus $V'$, cut into two solid tori by the annulus $A$.  On the right is a diagram depicting the construction of $X$ from Subcase 1.1 of the proof of Theorem \ref{thm:main}.}\label{fig:torus}
\end{figure}

We now consider two subcases, depending on how $h$ acts on $V'_1$ and $V'_2$.

\smallskip
\smallskip

\noindent \emph{Subcase 1.1:}  $h: V'_i \to V'_i$ for $i=1,2$.

Up to ambient isotopy, we may assume the following.

\begin{enumerate}
\item $K'(q) \cap V'_1 = K'(0) \cap V'_1$
\item $K'(q) \cap V'_2$ is obtained from $K'(0) \cap V'_2$ via $q$ full twists at $L'$
\end{enumerate}

Let $X$ be the 3-manifold obtained from $V'_2 - \eta(V'_2 \cap K'(0))$ by attaching to $A \subset \partial V'_2$ a thickened neighborhood of $\partial \eta(K'(0)) \cap V'_1$.  (See Figure \ref{fig:torus}.)  Then $h|_X$ fixes $X$ away from $V'_2$ and acts on $X \cap V'_2$ by twisting $\partial \eta(K'(0)) \subset \partial X$ $q$ times at $L'$.  Hence there is a homeomorphism from $X$ to $h(X)$ given by $q$ Dehn twists at $L' \subset \partial X$.  So, by Theorem \ref{thm:dehntw}, $L'$ bounds a disk in $X \subset (V' - \eta(K'))$.  But this means $L$ bounds a disk in $(V- \eta(K)) \subset M_K$ and hence $L$ is nugatory, contradicting our initial assumptions.

\smallskip
\smallskip

\noindent \emph{Subcase 1.2:}  $h$ maps $ V'_1 \to V'_2$ and $V'_2 \to V'_1$.

Again, we may assume the following.

\begin{enumerate}
\item $K'(q) \cap V'_1 = K'(0) \cap V'_2$
\item $K'(q) \cap V'_2$ is obtained from $K'(0) \cap V'_1$ via $q$ full twists at $L'$
\end{enumerate}

This time we construct $X$ from $V'_1 - \eta(V'_1 \cap K'(0))$ by attaching a thickened neighborhood of $\partial \eta(K'(0)) \cap V'_2$ to $A \subset \partial V'_1$.  Then the argument of Subcase 1.1 once again shows that $L$ must have been nugatory, giving a contradiction.

\smallskip

Hence, in Case 1, we have a pattern knot $K'$ for $K$ admitting an order-$q$ cosmetic generalized crossing change, as desired.

\smallskip
\smallskip

\noindent \emph{Case 2:} $T$ is compressible in $V - \eta(K(q))$.

In this case, $K(q)$ is contained in a 3-ball $B \subset V$.  Since $K(q)$ is not essential in $V$, by Lemma \ref{lem:LinV}, $K(0)$ is also not essential in $V$ and $K(0) = f(K'(0))$ can be isotoped to $K(q) = f(K'(q))$ via an isotopy contained in the 3-ball $B \subset V$.  This means that, once again, $K'(0)$ is isotopic to $K'(q)$ in $S^3$, and either $L'$ gives an order-$q$ cosmetic generalized crossing change for the pattern knot $K'$ or $L'$ is a nugatory crossing circle for $K'$.  Applying the arguments of each of the subcases above, we see that $L'$ cannot be nugatory, and hence $K'$ is a pattern knot for $K$ admitting an order-$q$ cosmetic generalized crossing change.
\end{proof}

In \cite{thurston}, Thurston shows that any knot falls into exactly one of three categories: torus knots, hyperbolic knots and satellite knots.  Theorem \ref{thm:main} gives obstructions to when cosmetic generalized crossing changes can occur in satellite knots.  This leads us to several useful corollaries.

\begin{cor}\label{cor:sat}
Let $K$ be a satellite knot admitting a cosmetic generalized crossing change of order $q$ with $|q| \geq 6$.  Then $K$ admits a pattern knot $K'$ which is hyperbolic.  
\end{cor}

\begin{proof}
Applying Theorem \ref{thm:main}, repeatedly if necessary, we know $K$ admits a pattern knot $K'$ which is not a satellite knot and which also admits an order-$q$ cosmetic generalized crossing change.  Kalfagianni has shown that fibered knots do not admit cosmetic generalized crossing changes of any order \cite{kalf}, and it is well-known that all torus knots are fibered.  Hence, by Thurston's classification of knots \cite{thurston}, $K'$ must be hyperbolic.  
\end{proof}

\begin{cor}\label{cor:fibered}
Suppose $K'$ is a torus knot.  Then no prime satellite knot with pattern $K'$ admits an order-$q$ cosmetic generalized crossing change with $|q| \geq 6$.
\end{cor}

\begin{proof}
Let $K$ be a prime satellite knot which admits a cosmetic generalized crossing change of order $q$ with $|q| \geq 6$.  By way of contradiction, suppose $T$ is a companion torus for $K$ corresponding to a pattern torus knot $K'$.  Since $K$ is prime, Lemma \ref{lem:type2orq<6} implies that $T$ is Type 1 and hence corresponds to a torus in $M = M_{K \cup L}$, which we will also denote by $T$.  If $T$ is not essential in $M$, then $T$ must to be parallel in $M$ to $\partial \eta(L)$.  But then $T$ would be compressible in $M_K$, which cannot happen since $T$ is a companion torus for $K \subset S^3$ and is thus essential in $M_K$.  So $T$ is essential in $M$ and there is a Haken system $\T$ for $M$ with $T \in \T$.  Since the pattern knot $K'$ is a torus knot and hence not a satellite knot, $T$ must be innermost with respect to $K$.  Then the arguments in the proof of Theorem \ref{thm:main} show that $K'$ admits an order-$q$ cosmetic generalized crossing change.  However, torus knots are fibered and hence admit no cosmetic generalized crossing changes of any order, giving us our desired contradiction.
\end{proof}

Note that if $K'$ is a torus knot which lies on the surface of the unknotted solid torus $V'$, then $(K', V')$ is a pattern for satellite knot which is by definition a cable knot.  Since any cable of a fibered knot is fibered, it was already known by \cite{kalf} that these knots do not admit cosmetic generalized crossing changes.  However, Corollary \ref{cor:fibered} applies not only to cables of non-fibered knots, but also to patten torus knots embedded in \emph{any} unknotted solid torus $V'$ and hence gives us a new class of knots which do not admit a cosmetic generalized crossing change of order $q$ with $|q| \geq 6$.


The proof of Corollary \ref{cor:fibered} leads us to the following.

\begin{cor}\label{cor:g1}
Let $K'$ be a knot such that $g(K')=1$ and $K'$ is not a satellite knot.  If there is a prime satellite knot $K$ such that $K'$ is a pattern knot for $K$ and $K$ admits a cosmetic generalized crossing change of order $q$ with $|q| \geq 6$, then $K'$ is hyperbolic and algebraically slice.
\end{cor}

\begin{proof}
By Corollary \ref{cor:fibered} and its proof, $K'$ is hyperbolic and admits a cosmetic generalized crossing change of order $q$.  Then by Lemma \ref{lem:g1}, $K'$ is algebraically slice.
\end{proof}

Finally, the following corollary summarizes the progress we have made in this paper towards answering Problem \ref{q:gncc}.

\begin{cor}\label{cor:main}
If there exists a knot admitting a cosmetic generalized crossing change of order $q$ with $|q| \geq 6$, then there must be such a knot which is hyperbolic.
\end{cor}

\begin{proof}
Suppose there is a knot $K$ with a cosmetic generalized crossing change of order $q$ with $|q| \geq 6$.  Since $K$ cannot be a fibered knot, $K$ is not  torus knot and either $K$ itself is hyperbolic, or $K$ is a satellite knot. If $K$ is a satellite knot, then by Corollary \ref{cor:sat} and its proof, $K$ admits a pattern knot $K'$ which is hyperbolic and has an order-$q$ cosmetic generalized crossing change.
\end{proof}

\bibliographystyle{amsplain}
\bibliography{references}

\end{document}